\documentclass[12pt,reqno]{amsart}

\title{Non-locally-free locus of O'Grady's ten dimensional example}
\author{Yasunari Nagai}

\address{
Graduate School of Mathematical Sciences, the University of Tokyo, 
3-8-1 Komaba Meguro Tokyo 153-8914, Japan}
\email{nagai@ms.u-tokyo.ac.jp}

\date{March 29, 2010}

\usepackage{amsthm}
\usepackage{amssymb}
\usepackage{amsxtra}
\usepackage[mathscr]{eucal}
\usepackage{enumerate}
\usepackage[all]{xy}
\usepackage[initials,alphabetic,lite]{amsrefs}

\usepackage{times,mathptmx}

\theoremstyle{plain}
\newtheorem{theorem}[subsection]{Theorem}

\newtheorem{proposition}[subsection]{Proposition}

\newtheorem*{theorem*}{Theorem}
\newtheorem*{corollary*}{Corollary}

\theoremstyle{definition}
\newtheorem{definition}[subsection]{Definition}

\theoremstyle{remark}
\newtheorem*{claim}{Claim}
\newtheorem{remark}[subsubsection]{Remark}

\DeclareSymbolFont{cmletters}{OML}{cmm}{m}{it}
\DeclareSymbolFont{cmsymbols}{OMS}{cmsy}{m}{n}
\DeclareSymbolFont{cmlargesymbols}{OMX}{cmex}{m}{n}
\DeclareMathSymbol{\myjmath}{\mathord}{cmletters}{"7C}
\DeclareMathSymbol{\myamalg}{\mathbin}{cmsymbols}{"71}
\DeclareMathSymbol{\mycoprod}{\mathop}{cmlargesymbols}{"60}
\let\jmath\myjmath
\let\amalg\myamalg
\let\coprod\mycoprod

\def\lto{\longrightarrow}

\DeclareMathOperator{\Supp}{Supp}

\DeclareMathOperator{\Pic}{Pic}
\DeclareMathOperator{\Hilb}{Hilb}
\DeclareMathOperator{\Sym}{Sym}

\DeclareMathOperator{\Proj}{Proj}
\DeclareMathOperator{\tr}{tr}

\DeclareMathOperator{\id}{id}
\DeclareMathOperator{\Aut}{Aut}

\DeclareMathOperator{\Hom}{Hom}
\DeclareMathOperator{\End}{End}

\DeclareMathOperator{\diag}{diag}

\DeclareMathOperator{\Ker}{Ker}

\DeclareMathOperator{\rank}{rank}
\DeclareMathOperator{\length}{length}

\DeclareMathOperator{\Quot}{Quot}

\def\git{/\!\!/}

\begin{document}

\baselineskip 17.6pt
\parskip 6pt

\maketitle
\vspace{-1cm}

\begin{abstract}
We give a completely explicit description of the fibers of 
the natural birational morphism from O'Grady's ten dimensional 
singular moduli space of sheaves on a K3 surface to the corresponding 
Donaldson--Uhlenbeck compactification. 
\end{abstract}

\section*{Introduction}

O'Grady's construction of a new example of irreducible symplectic manifold 
\cite{OG} had a big impact to the study 
of holomorphic symplectic manifold, as it has been understood to be difficult to 
construct an irreducible symplectic manifold that is not deformation 
equivalent to the examples given in \cite{B}. Since then, several mathematicians 
studied the properties of O'Grady's example. In the course of the study, 
it turned out to be important to understand the locus of non-locally-free sheaves 
on the singular moduli space that O'Grady considered, see e.g. \cites{R,P}. 
The locus is closely related to the topology, the singularities, and the geometry of 
O'Grady's example. Actually, O'Grady already considered 
the non-locally-free locus of the moduli space in his article \cite{OG} 
to show that the example he constructed has the different second Betti number 
than that of previously known examples. In this article, 
we obtain an explicit and more or less complete understanding of 
the non-locally-free locus of O'Grady's singular moduli space. 

Let $S$ be a projective K3 surface with Picard number one. 
What we mean by ``O'Grady's singular moduli space'' 
is the moduli space $M$ of Gieseker-semistable sheaves on $S$ of 
rank $2$, $c_1=0$, and $c_2=4$. 
By the general theory (\cite{H-L}, Definition-Theorem 8.2.8), 
we have a projective morphism 
\[
\varphi : M\to M^{DU}
\]
to the Donaldson-Uhlenbeck compactification $M^{DU}$. 
In our case, this turns out to be a birational morphism and 
the exceptional divisor $B$ of $\varphi$ is nothing but the non-locally-free locus 
of the moduli space $M$. Our Main Theorem (Theorem \ref{main}) gives an explicit description 
of the non-trivial fibers of the morphism $\varphi$.  

Theoretically, one can describe every fiber of the morphism $\varphi$ as 
a GIT quotient of an appropriate subvariety of a Quot scheme (see \S 1, see also \emph{op. cit.}, Chap. 8). 
However, such a description is not suitable for an explicit study; it is almost hopeless 
to really calculate the associated homogeneous coordinate ring, 
for the defining equation, or the Grothendieck-Pl\"ucker embedding, 
of a Quot scheme can already be very complicated.  

Instead of Quot scheme, we use a quiver-variety-like description of the fiber of $\varphi$ (\S 2). 
An obvious merit of this description is that one can really calculate the homogeneous invariant ring 
corresponding to the GIT quotient description of the fiber within the framework of classical 
invariant theory (\S3). Here we need to appeal for a brute force calculation using Gr\"obner basis 
to determine the relations among the generators of the invariant ring. 
We should also remark that the description cannot be globalized, 
meaning that one cannot obtain an explicit description of the whole non-locally-free locus $B$ by 
this method. 

\paragraph{\bf Acknowledgment}
The author would like to express his gratitude to Manfred Lehn for his suggestions 
and encouragements. He largely benefited from the discussions with him. 
He also thanks Arvid Perego for stimulating discussions. 
Most part of this work was done when the author was in Johannes Gutenberg-Universit\"at 
Mainz, where he was supported by SFB/TR 45 of DFG, Deutsche Forschungsgemeinschaft. 

\section{The non-locally-free locus}
Let $S$ be a K3 surface with $\Pic (S)=\mathbb Z[H]$, where $H$ is an ample divisor on $S$, 
and let $M$ be the moduli space of $H$-semistable torsion free coherent sheaves on $S$ with 
rank $2$, $c_1=0$, and $c_2=4$. If we denote the locus of strictly semistable sheaves 
by $\Sigma$, one can immediately know that $M$ is singular along $\Sigma$. Moreover, 
we have a stratification $\Sigma =\Sigma ^0\amalg \Sigma ^1$ by 
\[
\begin{aligned}
\Sigma ^0 &= \{[I_Z\oplus I_W] \mid Z, W\in \Hilb ^2(S),\; Z\neq W\}, \\
\Sigma ^1 &= \{[I_Z^{\oplus 2}] \mid Z\in \Hilb ^2(S)\}, 
\end{aligned}
\]
where the square brackets stand for the S-equivalence classes (\cite{OG}, Lemma 1.1.5). 

Let $B$ be the locus of non-locally free sheaves on $M$. Obviously, $B$ contains $\Sigma$. 
As is explained in the introduction, $B$ can be captured as the exceptional divisor of 
a projective birational morphism
\[
\varphi : M\to M^{DU}
\]
to the Donaldson-Uhlenbeck compactification (\cite{H-L}, Chap. 8). 
If $[E]\in B$, the double dual $E^{**}$ is a $\mu$-semistable sheaf and we have 
a short exact sequence
\[
0\lto E\lto E^{**} \lto Q(E)\lto 0
\]
with $Q(E)$ of length $c_2(E)-c_2(E^{**})$. We define the associated cycle $\gamma (Q)$ 
to an artinian coherent sheaf $Q$ by 
\[
\gamma (Q)=\sum _{p\in S} \length(Q_p)\cdot p \in \Sym ^l(S)\quad (l=\length (Q)).
\]
The correspondence $\varphi$ is, roughly speaking, given by 
\begin{equation}\label{DUcorresp}
E\mapsto (E^{**}, \gamma (Q(E)))
\end{equation}
(\cite{H-L}, Theorem 8.2.11). In our case, we have 
$E^{**}\cong \mathcal O_S^{\oplus 2}$ and $\length (Q(E))=4$ for every $[E]\in B$ 
(\cite{OG}, Proposition 3.1.1). This means that the first factor of the 
correspondence \eqref{DUcorresp} is always trivial, so that the restriction 
of $\varphi$ to $B$ is of the form
\[
\varphi _{|B} : B\to \Sym ^4(S); \quad E\mapsto \gamma(Q(E)).
\]
From this, we get a GIT description of $B$ as 
\[
B\cong \Quot (\mathcal O_S^{\oplus 2}, 4) \git SL(2),
\] 
where $\Quot (\mathcal O_S^{\oplus 2}, 4)$ stands for the Quot scheme of length 4 quotients 
of $\mathcal O_S^{\oplus 2}$ (cf. \cite{OG}, \S 3.4). 
We have a similar description of the fibers of $\varphi _{|B}$ as
follows. There is a natural morphism $\Quot (\mathcal O_S^{\oplus 2}, 4)\to \Sym ^4(S)$ 
by the correspondence 
\[
[\mathcal O_S^{\oplus 2}\to Q] \mapsto \gamma (Q).
\]
Denote by $\Quot _{\gamma}$ the fiber of this morphism over $\gamma \in \Sym^4(S)$. Then we have 
\begin{equation}\label{Quotquot}
\varphi _{|B}^{-1}(\gamma) \cong \Quot _{\gamma} \git SL(2).
\end{equation}
O'Grady proved that the general fiber of $\varphi _{|B}$, namely the case where $\gamma$ consists of 
four different points on $S$, is isomorphic to $\mathbb P^1$ (\emph{op. cit.}, Proposition 3.0.5). 
Note that from this fact, we know that $B$ is a codimension 1 subvariety of $M$, since $\dim M=10$. 

Let $\eta=[1^{\eta _1}, 2^{\eta _2}, 3^{\eta _3}, 4^{\eta _4}]$ be a partition of $4$, namely, 
$4=1\cdot \eta _1 + 2\cdot \eta _2 + 3\cdot \eta _3 + 4\cdot \eta _4$ with $\eta _i$ non-negative integers. 
Then, we have a natural stratification 
\[
\begin{aligned}
\Sym ^4(S) &=\coprod _{\eta} S^{(\eta)}, \\
S^{(\eta)} &=\left\{ \gamma = \sum _{m=1}^4 \sum _{i=1}^{\eta _i} m\cdot p_{m,i} 
\; \Bigg| \; p_{ij}\in S \text{ are distinct points} \right\}.
\end{aligned}
\]
A 0-cycle $\gamma \in \Sym ^4(S)$ is said to be of type $\eta$ if $\gamma \in S^{(\eta)}$. Let $B_{\gamma}$ be 
the fiber of $\varphi _{|B}:B\to \Sym ^4(S)$ \emph{with the reduced scheme structure}.

\begin{theorem}\label{main}
Let $\gamma\in \Sym ^4(S)$, $B_{\gamma}$, and $\Sigma$ as above. 

\begin{enumerate}[\rm(i)]

\item If $\gamma$ is of type $[1^4]$, $B_{\gamma}\cong \mathbb P^1$ and $(B_{\gamma}\cap \Sigma)_{red}$ is 
a three point set.
 
\item If $\gamma$ is of type $[1^2, 2]$, $B_{\gamma}\cong \mathbb P^2$ and $(B_{\gamma}\cap \Sigma)_{red}$ 
is a disjoint union of a line $l$ and a point $P$. More precisely, if $\gamma = p_1+p_2+2q$, 
\[
\begin{aligned}
l &=\{[I_{p_1+p_2}\oplus I_Z] \mid Z\in \Hilb ^2(S), \Supp Z=\{q\} \}\cong \mathbb P^1,\\
P&= [I_{p_1+q}\oplus I_{p_2+q}].
\end{aligned}
\]

\item If $\gamma$ is of type $[2^2]$, $B_\gamma$ is a quadric cone in $\mathbb P^4$. 
$(B_{\gamma}\cap \Sigma)_{red}$ is a disjoint union of a smooth hyperplane section
$T$ and the vertex $P$ of $B_\gamma$. More precisely, if $\gamma = 2p_1+2p_2$, 
\[
\begin{aligned}
T &=\{[I_{Z_1}\oplus I_{Z_2}] \mid Z_i\in \Hilb ^2(S), \Supp Z_i=\{p_i\}\, (i=1,2) \}\cong \mathbb P^1\times \mathbb P^1,\\
P&= [I_{p_1+p_2}^{\oplus 2}].
\end{aligned}
\]

\item If $\gamma$ is of type $[1,3]$, $B_\gamma$ is isomorphic to a quadric of rank 2 in $\mathbb P^4$, 
namely, the singular scroll in $\mathbb P^4$ that is the image of $\mathbb P_{\mathbb P^1}(\mathcal O^{\oplus 2}
\oplus \mathcal O(2))$. $(B_{\gamma}\cap \Sigma)_{red}$ is the line $l$ of vertices of $B_\gamma$ parametrizing 
\[
l=\{[I_{p+q}\oplus I_Z] \mid Z\in \Hilb ^2(S), \Supp Z=\{q\} \}\cong \mathbb P^1
\]
if $\gamma =p+3q$. 

\item If $\gamma$ is of type $[4]$, $B_\gamma$ is an irreducible divisor of the singular scroll
$\overline{\mathbb P}_{\mathbb P^1}(\mathcal O^{\oplus 3}\oplus \mathcal O(1)\oplus \mathcal O(4))
\subset \mathbb P^9$. 
The strict transform $B'_\gamma$ of $B_\gamma$ under the natural proper birational morphism 
\[
\mathbb P _{\mathbb P^1}(\mathcal O^{\oplus 3}\oplus \mathcal O(1)\oplus \mathcal O(4))\to 
\overline{\mathbb P}_{\mathbb P^1}(\mathcal O^{\oplus 3}\oplus \mathcal O(1)\oplus \mathcal O(4))
\]
is a topologically locally trivial family of quadric hypersurfaces of rank 2 
in the fiber $\mathbb P^4$ of $\mathbb P _{\mathbb P^1}(\mathcal O^{\oplus 3}\oplus \mathcal O(1)\oplus \mathcal O(4))$. 
$(B_{\gamma}\cap \Sigma)_{red}$ is the plane $\Pi =\overline{\mathbb P}_{\mathbb P^1}(\mathcal O^{\oplus 3})$ 
parametrizing 
\[
\Pi =\{[I_{Z_1}\oplus I_{Z_2}] \mid Z_i\in \Hilb ^2(S), \Supp Z_i=\{p\}\, (i=1,2) \} \cong \mathbb P^2 
\]
if $\gamma = 4p$. Moreover, the envelope $\Delta$ of the family of lines on $\Pi$ that is the image of 
the vertex lines of the fibers of $B'_{\gamma}$ is a non-singular quadric on $\Pi$ parametrizing
\[
\Delta = (B_\gamma\cap \Sigma ^1)_{red}
=\{[I_Z^{\oplus 2}] \mid Z\in \Hilb ^2(S), \Supp (Z)=\{p\} \}.
\]
\end{enumerate}
\end{theorem}

The rest of the article is devoted to the proof of this theorem. 
Note that the first assertion (i) is nothing but Proposition 3.0.5 of \cite{OG}. 

\begin{remark}
From a view to the geometry of irreducible symplectic manifold, 
it should be more interesting to look at the strict transform $\widetilde B_{\gamma}$ 
of $B_{\gamma}$ under O'Grady's resolution $\pi: \widetilde M\to M$. 
According to \cite{L-S}, O'Grady's symplectic resolution $\pi$ is nothing 
but the blowing-up along the locus of strictly semistable sheaves $\Sigma _{red}$. 
However, it is not so easy to analyze the induced blowing-up  
$\widetilde B_{\gamma}=Bl _{B_{\gamma}\cap (\Sigma _{red})} B_{\gamma} \to B_{\gamma}$, 
since the scheme theoretic intersection $B_{\gamma}\cap (\Sigma _{red})$ tends to have 
highly non-trivial non-reduced scheme structure over deeper strata. 
\end{remark}

\section{GIT description of the fiber of $\varphi$}

The description \eqref{Quotquot} of $B_{\gamma}$ by Quot scheme is not 
suitable to prove Theorem \ref{main}, since $\Quot _{\gamma}$ already has a complicated 
projective geometry so that it is hopeless to put the GIT quotient involved in \eqref{Quotquot} 
into an actual calculation. Instead, we use a quivar-variety-like description of $B_{\gamma}$, 
which can be seen as an analogy of the description of $\Hilb ^n(\mathbb C^2)$ as a hyper-K\"ahler 
quotient (see, for example, \cite{N} Chap. 2 \& 3, in particular, Theorem 3.24). 

\begin{definition}
Let $V=\mathbb C^2$. 
For a 0-cycle $\gamma =\sum m_i p_i \in \Sym ^4(S)$, 
we define a sheaf of $\mathbb C$-vector spaces $\mathcal Q_{\gamma}$ 
of finite length by 
\[
\mathcal Q_{\gamma, x}=
\begin{cases}
\mathbb C^{m_i} & \text{if } x=p_i,\\
0 & \text{otherwise} .
\end{cases}		   
\]
Define 
\[
N_m=\{(A,B)\in \mathfrak{sl} (\mathbb C^m)^{\oplus 2}
\mid [A, B]=0,\; A^iB^j=O \text{ for } i+j=m\}, 
\]
and define $N_{\gamma}$ by
\[
N_{\gamma}=N_1^{\eta _1}\times N_2^{\eta _2}\times
N_3^{\eta _3}\times N_4^{\eta _4}
\]
if $\gamma$ is of type $\eta=[1^{\eta _1}, 2^{\eta _2}, 3^{\eta _3}, 4^{\eta _4}]$.
Here we note that $N_1$ is a reduced one point set. 
We define $Y_{\gamma}$ by
\[
Y_{\gamma}=\Hom _{\mathbb C}(V,\Gamma (\mathcal Q_{\gamma}))\times N_{\gamma}, 
\]
and define $G_{\gamma}=GL(V)\times \Aut (\mathcal Q_{\gamma})$. 
We always let $\Aut (\mathcal Q_{\gamma})$ act on 
$N_{\gamma}$ by the adjoint action at each point 
of the support of $\mathcal Q_{\gamma}$. 
$G_\gamma$ can always be seen as a subgroup of 
$GL(V)\times GL(\mathbb C^4)$ by diagonal embedding. 
We define a character $\chi$ by
\[
\chi =(\det {}_V)^{-2}\cdot (\det {}_{\mathbb C^4}) : GL(V)\times GL(\mathbb C^4)
\to \mathbb C^*
\]
and let $\chi _{\gamma}$ be the composition 
\[
\chi _{\gamma}: G_{\gamma}\hookrightarrow GL(V)\times GL(\mathbb C^4)
\overset{\chi}{\to} \mathbb C^*.
\]
\end{definition}

\begin{theorem}\label{quiverlike}
Notation as above. We have an isomorphism 
\[
B_\gamma \cong Y_\gamma \underset{\chi _{\gamma}}{\git} G_\gamma
= \Proj \left( \bigoplus _{n=0}^{\infty} A(Y_{\gamma})^{G_{\gamma},\chi _{\gamma}^n} \right), 
\]
where $A(Y_{\gamma})$ is the affine coordinate ring of $Y_{\gamma}$ 
and $A(Y_{\gamma})^{G_{\gamma},\chi _\gamma^n}$ is the vector subspace consisting of
$G_{\gamma}$-semi-invariants of $A(Y_{\gamma})$ whose character is $\chi_{\gamma}^n$.
\end{theorem}

\begin{remark}\label{toQuot}
Let 
\[
\Psi =(\psi,\alpha)\in Y_{\gamma}
=\Hom (V,\Gamma (\mathcal Q_{\gamma}))\times N_{\gamma}.
\]
Since $N_{\gamma}$ parametrizes $\mathcal O_S$-module structures on 
$\mathcal Q_{\gamma}$, $\Psi$ defines \emph{a morphism of 
$\mathcal O_S$-modules}
\[
V\otimes \mathcal O_S\to \mathcal Q_{\gamma}.
\]
We denote this morphism also by $\Psi$ slightly abusing the notation. 
In particular, if it is surjective, $\Psi$ defines a point 
$\overline \Psi\in \Quot _{\gamma}\subset \Quot (V\otimes \mathcal O_S,4)$.
Moreover, every point in the $\Aut (\mathcal Q_{\gamma})$-orbit of $\Psi$ 
corresponds to $\overline \Psi$ and we can prove 
\[
(\Quot _{\gamma})_{red}\cong Y_{\gamma} \underset{\det _{\mathcal Q_{\gamma}}}{\git}
\Aut (\mathcal Q_{\gamma}). 
\]
\end{remark}

\begin{remark}
$\chi _{\gamma}$ is the character associated to the determinant line bundle 
for the universal family over the Quot scheme. It will turn out that $\chi _{\gamma}$ is 
the only meaningful polarization in the calculation of invariant rings in \S 3. 
\end{remark}

\begin{remark}
There is nothing to do with the condition $c_2=4$ in 
the description of the fiber $B_{\gamma}$ of the morphism to Donaldson-Uhlenbeck 
compactification in Theorem \ref{quiverlike}. We can easily generalize the theorem 
to the case of moduli space $M_S(2,0,2k)$ of semistable sheaves 
of rank $2$, $c_1=0$, and $c_2=2k$ for every positive integer $k$. 
\end{remark}

By the standard theory of GIT construction of a moduli space, 
the theorem is a consequence of the following

\begin{proposition}\label{stability}
Notation as above. Let $\Psi=(\psi, \alpha)\in Y_{\gamma}$. Then, the following are equivalent:
\begin{enumerate}[\rm(i)]
\item The morphism of $\mathcal O_S$-modules $\Psi : V\otimes \mathcal O_S\to \mathcal Q_{\gamma}$ is surjective and 
$E=\Ker \Psi$ is semistable (resp. stable).

\item $\Psi$ is surjective and for every one dimensional subspace $W\subset V$, 
$\dim \Psi(W\otimes O_S) \geqslant 2$ (resp. $\dim \Psi(W\otimes O_S) > 2$\,). 

\item $\Psi$ is a $(G_{\gamma},\chi)$-semistable (resp. $(G_{\gamma},\chi)$-stable) point. 
\end{enumerate}
\end{proposition}

\begin{proof}
Noting that $E=\Ker \Psi$ is automatically $\mu$-semistable and 
any destabilizing subsheaf $F$ of $E$ corresponds to a one dimensional 
subspace $W\subset V$ by $F^{**}=W\otimes \mathcal O_S$, 
the equivalence of (i) and (ii) is nothing but Lemma 1.1.5 of \cite{OG}.

The equivalence of (ii) and (iii) is just an application of 
Hilbert-Mumford's numerical criterion. We demonstrate the case 
where $\gamma$ is of type $[1^2,2]$, as the proofs in the other cases 
are similar. A reader can skip this part and go promptly to \S 3, for 
the argument here is classical and standard. 

\noindent \stepcounter{subsubsection} (\thesubsubsection)\; 
Let $\gamma =p_1+p_2+2q$ and $\mathcal Q_{\gamma}
\cong \mathcal O_{p_1}\oplus \mathcal O_{p_2}\oplus \mathcal Q_{\gamma, q}$ 
with $\mathcal Q_{\gamma, q}\cong \mathbb C^2$. Then, 
\[
\begin{aligned}
Y_{\gamma} &\cong 
\Hom (V,\mathbb C)\times \Hom (V,\mathbb C)\times 
\Hom (V, \mathbb C^2)\times N_2,\\
G_{\gamma} &=
GL(V)\times \mathbb C^*\times \mathbb C^*\times GL(\mathbb C^2),\\
\chi _{\gamma} &=
(\det {}_V)^{-2}\cdot \id _{\mathbb C^*}\cdot \id _{\mathbb C^*}\cdot (\det {}_{\mathbb C^2}).
\end{aligned}
\]
For a one parameter subgroup (1-PS, in short) $\lambda :\mathbb C^*\to G_{\gamma}$, 
we define the pairing $\langle \chi _{\gamma}, \lambda\rangle =m$ by $\chi _{\gamma}\circ \lambda (t)=t^m$ for 
$t\in \mathbb C^*$. 
According to a form of the numerical criterion (\cite{K}, Proposition 2.5), $\Psi \in Y_{\gamma}$ 
is $(G_{\gamma}, \chi _{\gamma})$-semistable if and only if 
$\langle \chi _{\gamma},\lambda\rangle\geqslant 0$ holds 
for every 1-PS $\lambda$ of $G_{\gamma}$ such that 
the limit $\lim _{t\to 0}\lambda (t)\cdot \Psi$ converges. 
Moreover, a semistable point 
$\Psi$ is stable if and only if the strict inequality $\langle \chi _{\gamma},\lambda\rangle > 0$ holds 
for every \emph{non-trivial} 1-PS $\lambda$ of $G_{\gamma}$ such that 
the limit $\lim _{t\to 0}\lambda (t)\cdot \Psi$ converges. 

\noindent \stepcounter{subsubsection} (\thesubsubsection)\; 
Represent a point $\Psi\in Y_{\gamma}$ by matrices
\[
\Psi=\left( (x_1\;\; x_2) , (y_1\;\; y_2), 
Z=\begin{pmatrix}
z_{11} & z_{12}\\
z_{21} & z_{22}
\end{pmatrix} ; \,
A=\begin{pmatrix}
a_{11} & a_{12}\\
a_{21} & a_{22}
\end{pmatrix} , 
B=\begin{pmatrix}
b_{11} & b_{12}\\
b_{21} & b_{22}
\end{pmatrix} \right)
\]
using bases for $V$ and $\mathcal Q_{\gamma}$. 
For $\mathbf r=(r_1,\cdots ,r_6)\in \mathbb Z^6$, we define a 1-PS
$\lambda _{\mathbf r}:\mathbb C^*\to G_{\gamma}$ by
\[
\lambda _{\mathbf r}(t)
=\left( \diag (t^{r_1},t^{r_2}), t^{r_3}, t^{r_4}, \diag(t^{r_5}, t^{r_6}) \right). 
\]
Then, we have
\begin{multline*}
\lambda _{\mathbf r} (t)\cdot \Psi
=\left( (t^{r_3-r_1}\cdot x_1\quad t^{r_3-r_2}\cdot x_2) , (t^{r_4-r_1}\cdot y_1\quad t^{r_4-r_2}\cdot y_2), 
\right. \\
\begin{pmatrix}
t^{r_5-r_1}\cdot z_{11} & t^{r_5-r_2}\cdot z_{12}\\
t^{r_6-r_1}\cdot z_{21} & t^{r_6-r_2}\cdot z_{22}
\end{pmatrix} ; \\
\left. 
\begin{pmatrix}
a_{11} & t^{r_5-r_6}\cdot a_{12}\\
t^{-r_5+r_6}\cdot a_{21} & a_{22}
\end{pmatrix} , 
\begin{pmatrix}
b_{11} & t^{r_5-r_6}\cdot b_{12}\\
t^{-r_5+r_6}\cdot b_{21} & b_{22}
\end{pmatrix} \right) .
\end{multline*}
Note that any 1-PS $\lambda$ is equivalent to $\lambda _{\mathbf r}$ 
for some $\mathbf r$ up to a base change of $V$ and $\mathcal Q_{\gamma}$, 
and that 
\[
\langle \chi _{\gamma}, \lambda _{\mathbf r}(t)\rangle
= -2r_1-2r_2+r_3+r_4+r_5+r_6. 
\]

\noindent \stepcounter{subsubsection} (\thesubsubsection)\; 
First, we prove (iii)$\Rightarrow$(ii). 
Assume that $\Psi$ is not surjective. This is equivalent to assume either
(a) $(x_1\quad x_2)=0$ or $(y_1\quad y_2)=0$, or (b) $\rank Z\leqslant 1$ and $AZ=BZ=O$. 
In the case (a), say, if we have $(x_1\quad x_2)=0$, $\lambda _{\mathbf r}$ for $\mathbf r=(0,0,-1,0,0,0)$ 
is destabilizing 1-PS, namely, $\lambda _{\mathbf r}(t)\cdot \Psi$ has a limit as $t\to 0$, 
but $\langle \chi _{\gamma},\lambda _{\mathbf r}\rangle =-1<0$. Consider the case (b). If $Z=O$, 
$\lambda _{\mathbf r}$ for $\mathbf r=(0,0,0,0,-1,-1)$ is destabilizing.  If $\rank Z=1$, 
we may assume $Z=\begin{pmatrix} 0 & 0 \\ 1 & 0 \end{pmatrix}$ taking suitable bases for $V$ 
and $\mathcal Q_{\gamma}$. 
$AZ=BZ=O$ implies $a_{12}=b_{12}=0$. Then, 
$\lambda _{\mathbf r}$ for $\mathbf r=(0,0,0,0,-1,0)$ is destabilizing. 
This proves that $\Psi$ is surjective if $\Psi$ is $(G_{\gamma}, \chi _{\gamma})$-semistable. 

\noindent \stepcounter{subsubsection} (\thesubsubsection)\; 
Now we assume that $\dim \Psi(W\otimes \mathcal O_S)<2$
for $W=
\mathbb C\cdot 
\begin{pmatrix}
0 \\ 1
\end{pmatrix}$. This is equivalent to say that
\begin{enumerate}[(a)]
\item $z_{12}=z_{22}=0$ and one of $x_2$ or $y_2$ is zero, or 
\item $x_2=y_2=0$, $\begin{pmatrix} z_{12} \\ z_{22}\end{pmatrix}\neq 0$, and
$A\begin{pmatrix} z_{12} \\ z_{22}\end{pmatrix}=B\begin{pmatrix} z_{12} \\ z_{22}\end{pmatrix}=0$. 
\end{enumerate}
In the case (a), say, $x_2=z_{12}=z_{22}=0$, $\lambda _{\mathbf r}$ for $\mathbf r=(0,1,0,1,0,0)$ is 
destabilizing. In the case (b), we may assume $\begin{pmatrix} z_{12} \\ z_{22}\end{pmatrix}
=\begin{pmatrix} 0 \\ 1\end{pmatrix}$. 
Then, $A\begin{pmatrix} z_{12} \\ z_{22}\end{pmatrix}=B\begin{pmatrix} z_{12} \\ z_{22}\end{pmatrix}=0$ 
implies $a_{12}=b_{12}=0$. Therefore, $\lambda _{\mathbf r}$ for $\mathbf r=(0,1,0,0,0,1)$ is destabilizing. 

\noindent \stepcounter{subsubsection} (\thesubsubsection)\label{strss}\; 
Assume that $\dim \Psi(W\otimes \mathcal O_S)=2$
for $W=
\mathbb C\cdot 
\begin{pmatrix}
0 \\ 1
\end{pmatrix}$, i.e., assume either 
\begin{enumerate}[(a)]
\item $z_{12}=z_{22}=0$, $x_2\neq 0$, and $y_2\neq 0$, or
\item Only one of $x_2$ or $y_2$ is zero, $\begin{pmatrix} z_{12} \\ z_{22}\end{pmatrix}\neq 0$, and
$A\begin{pmatrix} z_{12} \\ z_{22}\end{pmatrix}=B\begin{pmatrix} z_{12} \\ z_{22}\end{pmatrix}=0$, or
\item $x_2=y_2=0$, $\begin{pmatrix} z_{12} \\ z_{22}\end{pmatrix}\neq 0$, and
at least one of $A\begin{pmatrix} z_{12} \\ z_{22}\end{pmatrix}$ or 
$B\begin{pmatrix} z_{12} \\ z_{22}\end{pmatrix}$ is not zero. 
\end{enumerate}
In the case (a), $\lambda _{\mathbf r}$ for $\mathbf r=(0,1,1,1,0,0)$ is ``strictly destabilizing'', i.e., 
$\lambda _{\mathbf r}(t)\cdot \Psi$ has a limit as $t\to 0$, but 
$\langle \chi _{\gamma},\lambda _{\mathbf r}\rangle =0 \leqslant 0$. Similarly, in the case (c), 
$\lambda _{\mathbf r}$ for $\mathbf r=(0,1,0,0,1,1)$ is strictly destabilizing. In the case (b), 
we may assume $x_2=0$ and $\begin{pmatrix} z_{12} \\ z_{22}\end{pmatrix}
=\begin{pmatrix} 0 \\ 1\end{pmatrix}$. Then, as before, we have $a_{12}=b_{12}=0$ and 
$\lambda _{\mathbf r}$ for $\mathbf r=(0,1,0,1,0,1)$ is strictly destabilizing. 
This finishes the proof of (iii)$\Rightarrow$(ii). 

\noindent \stepcounter{subsubsection} (\thesubsubsection)\; 
Now we prove the converse (ii)$\Rightarrow$(iii). Suppose that $\Psi$ is 
not $(G_{\gamma}, \chi_{\gamma})$-semistable (resp. stable). 
Then, up to a change of bases for $V$ and 
$\mathcal Q_{\gamma}$, there exists some non-zero $\mathbf r\in \mathbb Z^6$ such that 
$\lim _{t\to 0} \lambda _{\mathbf r}(t)\cdot \Psi$ exists but 
$\langle \chi _{\gamma}, \lambda _{\mathbf r} \rangle <0$ (resp. $\leqslant 0$). 
Let us call this $\mathbf r$ a destabilizing vector. Note that we may assume 
\begin{equation}\label{cst}
0=r_1 \leqslant r_2, \quad r_3\leqslant r_4, \quad r_5 \leqslant r_6
\end{equation}
by symmetry. Under this condition, $\langle \chi _{\gamma}, \lambda _{\mathbf r} \rangle 
<0$ (resp. $\leqslant 0$) is equivalent to 
\begin{equation}\label{destabvec}
2 r_2>\text{ (resp. $\geqslant$) }r_3+r_4+r_5+r_6. 
\end{equation}

\noindent \stepcounter{subsubsection} (\thesubsubsection)\; 
Assume that $\mathbf r$ with $r_3<0$ is a destabilizing vector for $\Psi$. 
Then, we must have $x_1=x_2=0$, and therefore $\Psi$ cannot be surjective. 
Assume that $\mathbf r$ with $r_5<0$ is a destabilizing vector for $\Psi$.
Then we have $z_{11}=z_{12}=0$. Furthermore, if we even have $r_6<0$, 
$z_{21}$ and $z_{22}$ must vanish, therefore $\Psi$ cannot be surjective. 
So, let us assume $r_6 \geqslant 0$. 
Then we automatically have $r_6>r_5$, so that $a_{12}=b_{12}=0$. 
Noting that $A^2=O=B^2$, we have $a_{11}=a_{22}=b_{11}=b_{22}=0$. Thus, 
we must have $AZ=BZ=O$ and $\Psi$ cannot be surjective. 

\noindent \stepcounter{subsubsection} (\thesubsubsection)\; 
Now let us assume $\Psi$ to be surjective. This implies that $r_3, r_4, r_5, r_6 \geqslant 0$. 
For a destabilizing vector $\mathbf r$, we deduce $r_2>\text{(resp. $\geqslant$) }0$ from \eqref{destabvec}. 
If $r_3, r_4, r_5, r_6<r_2$ for a destabilizing $\mathbf r$, 
we get $x_2=y_2=z_{12}=z_{22}=0$ and 
$\dim \Psi(W\otimes \mathcal O_S)=0<2$
for $W=
\mathbb C\cdot 
\begin{pmatrix}
0 \\ 1
\end{pmatrix}$.

\noindent \stepcounter{subsubsection} (\thesubsubsection)\; 
Suppose $\mathbf r$ destabilizing such that exactly 
three of $r_3, r_4, r_5, r_6$ are less than $r_2$, 
i.e., $r_3, r_4, r_5 < r_2$ or $r_3, r_5, r_6 < r_2$. In the first case, 
we have $x_2=y_2=z_{12}=0$. Moreover $a_{12}=b_{12}=0$, since we must have 
$r_5<r_6$. Therefore, $\dim \Psi(W\otimes \mathcal O_S)=1<2$
for $W=
\mathbb C\cdot 
\begin{pmatrix}
0 \\ 1
\end{pmatrix}$. Also in the second case, we must have $x_2=z_{12}=z_{22}=0$, 
which implies that $\dim \Psi(W\otimes \mathcal O_S)=1<2$
for $W=
\mathbb C\cdot 
\begin{pmatrix}
0 \\ 1
\end{pmatrix}$. 

\noindent \stepcounter{subsubsection} (\thesubsubsection)\; 
Now we assume that $\mathbf r$ destabilizing with exactly two of $r_3, r_4, r_5, r_6$ are
less than $r_2$. Then we must have, say, $r_3, r_4\geqslant r_2$. But under the strict inequality 
\eqref{destabvec}, we get 
\[
2 r_2>r_3+r_4+r_5+r_6\geqslant 2r_2+r_5+r_6, 
\]
which is contradiction since $r_5, r_6\geqslant 0$. This proves (ii)$\Rightarrow$(iii) in the 
semistable case. 

\noindent \stepcounter{subsubsection} (\thesubsubsection)\; 
Finally, assume that $\Psi$ is strictly $(G_{\gamma}, \chi _{\gamma})$-semistable. 
In this case, we can deduce from the argument in the previous paragraph that 
a ``strictly destabilizing'' vector $\mathbf r$ should satisfy that 
two of $r_3, r_4, r_5, r_6$ are equal to $r_2$ and the others are 0, namely 
\begin{enumerate}[(a)]
\item $r_3=r_4=r_2$ and $r_5=r_6=0$, or
\item $r_4=r_6=r_2$ and $r_3=r_5=0$, or
\item $r_5=r_6=r_2$ and $r_3=r_4=0$. 
\end{enumerate}
It is immediate that we have the corresponding statement in (2.3.5) in each case, 
and therefore $\dim \Psi(W\otimes \mathcal O_S)=2$
for $W=
\mathbb C\cdot 
\begin{pmatrix}
0 \\ 1
\end{pmatrix}$. This completes the proof of Proposition \ref{stability}. 
\end{proof}

\section{Calculation of the invariant rings}

In this section, we actually calculate the homogeneous ring of semi-invariants
\begin{equation}\label{homoginvring}
\mathscr R_{\gamma} = \bigoplus _{n=0}^{\infty} A(Y_{\gamma})^{G_{\gamma}, \chi _{\gamma}^n}
\end{equation}
appeared in Theorem \ref{quiverlike} for each partition type $\eta$ and finish the proof of 
Theorem \ref{main}. 

\subsection{} First, we treat the case where $\gamma$ is of type $[1^4]$. We include this case 
because this is the toy example for the calculation of all the remaining cases. In this case, 
\[
\begin{aligned}
Y_{\gamma} &\cong 
\prod _{i=1}^4 \Hom (V,\mathbb C),\\
G_{\gamma} &=GL(V)\times (\mathbb C^*)^4,\\
\chi _{\gamma} &=
(\det {}_V)^{-2}\cdot \id _{\mathbb C^*}\cdot \id _{\mathbb C^*}\cdot \id _{\mathbb C^*}\cdot \id _{\mathbb C^*}, 
\end{aligned}
\]
and the space $A(Y_{\gamma})^{G_{\gamma}, \chi _{\gamma}^n}$ consists of 
$SL(V)$-invariants whose character is $\chi _{\gamma}^n$. If we write $\Psi\in Y_{\gamma}$ as
\[
\Psi = 
\begin{pmatrix}
x_1 & x_2\\
y_1 & y_2\\
z_1 & z_2\\
w_1 & w_2
\end{pmatrix}
\]
using a coordinate, the ring of $SL(V)$-invariants is generated by the Pl\"ucker coordinates, 
namely, the $(2\times2)$-subdeterminants
\begin{multline*}
p_{12}=\begin{vmatrix} x_1 & x_2 \\ y_1 & y_2\end{vmatrix},\; 
p_{13}=\begin{vmatrix} x_1 & x_2 \\ z_1 & z_2\end{vmatrix},\;
p_{14}=\begin{vmatrix} x_1 & x_2 \\ w_1 & w_2\end{vmatrix},\;\\
p_{23}=\begin{vmatrix} y_1 & y_2 \\ z_1 & z_2\end{vmatrix},\; 
p_{24}=\begin{vmatrix} y_1 & y_2 \\ w_1 & w_2\end{vmatrix},\; 
p_{34}=\begin{vmatrix} z_1 & z_2 \\ w_1 & w_2\end{vmatrix},
\end{multline*}
subject to the Pl\"ucker relation
\[
p_{12}p_{34}-p_{13}p_{24}+p_{14}p_{23}=0, 
\]
which is the only relation between the Pl\"ucker coordinates. If we calculate the weights for these $SL(V)$-invariants, 
we get
\begin{center}
\begin{tabular}{c||c|c|c|c|c}
   & $\det {}_V$ & $\mathbb C^*$ & $\mathbb C^*$ & $\mathbb C^*$ & $\mathbb C^*$ \\
\hline
$p_{12}$ & -1 & 1 & 1 & 0 & 0\\
$p_{13}$ & -1 & 1 & 0 & 1 & 0\\
$p_{14}$ & -1 & 1 & 0 & 0 & 1\\
$p_{23}$ & -1 & 0 & 1 & 1 & 0\\
$p_{24}$ & -1 & 0 & 1 & 0 & 1\\
$p_{34}$ & -1 & 0 & 0 & 1 & 1\\ \hline
$\chi_{\gamma}$ & -2 & 1 & 1 & 1 & 1
\end{tabular}
\end{center}
Therefore, the space $A(Y_{\gamma})^{G_{\gamma},\chi_{\gamma}}$ is spanned by
\[
u_0=p_{12}p_{34},\quad u_1=p_{13}p_{24},\quad  u_2=p_{14}p_{23}
\]
and we have a description of the homogenous ring $\mathscr R_{\gamma}$ \eqref{homoginvring} as
\[
\mathscr R_{\gamma}=\bigoplus _{n=0}^{\infty} A(Y_{\gamma})^{G_{\gamma}, \chi _{\gamma}^n}
\cong \mathbb C[u_0,u_1,u_2] / (u_0-u_1+u_2).
\]
Applying Theorem \ref{quiverlike}, 
this leads to the conclusion $B_{\gamma}\cong \Proj \mathscr R_{\gamma}\cong \mathbb P^1$. 
Moreover, a point $\Psi$ is strictly semistable if and only if one of the Pl\"ucker coordinates vanishes 
(Proposition \ref{stability}). Therefore, we get
\[
(B_{\gamma}\cap \Sigma)_{red}
\cong (u_0-u_1+u_2=0, \; u_0u_1u_2=0) = \{\text{ 3 points }\} \subset \mathbb P^2.
\]

\subsection{}
Let $\gamma$ be of type $[1^2, 2]$ and $W=\mathbb C^2$. Then, 
\[
\begin{aligned}
Y_{\gamma} &\cong 
\Hom (V,\mathbb C)\times \Hom (V,\mathbb C)\times 
\Hom (V, W)\times N_2,\\
G_{\gamma} &=
GL(V)\times \mathbb C^*\times \mathbb C^*\times GL(W),\\
\chi _{\gamma} &=
(\det {}_V)^{-2}\cdot \id _{\mathbb C^*}\cdot \id _{\mathbb C^*}\cdot (\det {}_W).
\end{aligned}
\]
Just as in \S 3.1, first we calculate the invariant ring $A(Y_{\gamma})^{SL(V)\times SL(W)}$, and then 
pick up the homogeneous subring of elements with characters $\chi _{\gamma}^n,\; (n\geqslant 0)$.  
Let us write $\Psi=(\psi,\alpha)\in Y_{\gamma}=\Hom (V, \Gamma (\mathcal Q_{\gamma}))\times N_{\gamma}$ as 
\[
\Psi=\Big( \psi = \left(\begin{array}{@{\,}cc@{\,}}
				x_1 & x_2 \\ \hline
				y_1 & y_2 \\ \hline
				z_{11} & z_{12}\\
				z_{21} & z_{22}
				\end{array}\right), \alpha = (A,B) \Big), \quad (\, (A,B)\in N_2)
\]
using a coordinate. 
As $GL(V)$ acts trivially on $N_2$, 
the generators  for $SL(V)$-invariant ring $A(Y_{\gamma})^{SL(V)}$ is always given by Pl\"ucker 
for the coordinates of $\psi$, plus the entries of $A$ and $B$. Let 
\[
f_1 =\begin{vmatrix} x_1 & x_2 \\ y_1 & y_2\end{vmatrix},\;\,
f_2 =\begin{vmatrix} z_{11} & z_{12} \\ z_{21} & z_{22}\end{vmatrix},\;\, 
w_1 =\begin{pmatrix}
	\begin{vmatrix} x_1 & x_2 \\ z_{11} & z_{12}\end{vmatrix} \vspace{6pt}\\
	\begin{vmatrix} x_1 & x_2 \\ z_{21} & z_{22}\end{vmatrix}\end{pmatrix},\;\, 
w_2 =\begin{pmatrix}
	\begin{vmatrix} y_1 & y_2 \\ z_{11} & z_{12}\end{vmatrix} \vspace{6pt}\\
	\begin{vmatrix} y_1 & y_2 \\ z_{21} & z_{22}\end{vmatrix}\end{pmatrix}. 
\]
$SL(W)$ acts trivially on $f_1$ and $f_2$, and by left multiplication on $w_1$ and $w_2$. 
Here, the situation came out is exactly the problem of invariants for mixed tensors. 
Classical invariant theory had a method to deal with this kind of problem, the \emph{symbolic method}
(\cite{PV}, Theorem 9.3). The statement relevant to the current situation is the following:

\begin{quote}
Let $W$ be a vector space and $U=W^{\oplus 2}\oplus \End (W)^{\oplus 2}$. Let $SL(W)$ act on $U$ 
by left multiplication on the factors $W$ and by conjugation on the factors $\End (W)$. Represent an element 
of $Y$ by $(w_1,w_2;A,B)$. Then, the ring of $SL(W)$ invariants $A(U)^{SL(W)}$ is generated by 
\begin{itemize}
\item $\tr (W_0(A,B))$, and
\item $\det (W_1(A,B)w_i \mid W_2(A,B)w_j)\quad (i,j\in \{1,2\})$, 
\end{itemize}
where $W_i(A,B)\; (i=0,1,2)$ stands for an arbitrary word in $A$ and $B$.
\end{quote}

\noindent
This claim is not strong enough in general, because it gives only infinite set of generators. 
The symbolic method say nothing about a bound on the generating set, whereas it should be 
actually finite by the famous theorem of Hilbert. However, in our situation, the endomorphisms $A$ and $B$ are
\emph{nilpotent} and \emph{commute each other}, so that the set of non-zero words in $A$ and $B$ is effectively bounded, 
and so is the generating set of invariants. More concretely, every $SL(V)\times SL(W)$-invariants in our case is a polynomial in 
\[
\begin{aligned}
\xi_1&=f_1,\quad \xi_2=f_2,\\
\xi_3&= \det (Aw_1\mid w_1),\; \xi_4=\det (Bw_1 \mid w_1),\\
\xi_5&= \det (Aw_1\mid w_2),\; \xi_6=\det (Bw_1 \mid w_2),\\
\xi_7&= \det (Aw_2\mid w_2),\; \xi_8=\det (Bw_2 \mid w_2) 
\end{aligned}
\]
(Note that $\det (w_1\mid w_2)=f_1\cdot f_2$).
The weights of these invariants with respect to the characters are given by the following table.
\begin{center}
\begin{tabular}{c||c|c|c|c}
 & $\det {}_V$ & $\mathbb C^*\; (y)$ & $\mathbb C^*\; (z)$ & $\det {}_W$ \\ \hline
 $\xi _1$ & -1 & 0 & 0 & 1\\
 $\xi _2$ & -1 & 1 & 1 & 0\\
 $\xi _3$ & -2 & 2 & 0 & 1\\
 $\xi _4$ & -2 & 2 & 0 & 1\\
 $\xi _5$ & -2 & 1 & 1 & 1\\
 $\xi _6$ & -2 & 1 & 1 & 1\\
 $\xi _7$ & -2 & 0 & 2 & 1\\
 $\xi _8$ & -2 & 0 & 2 & 1\\ \hline
 $\chi_{\gamma}$ & -2 & 1 & 1 & 1
\end{tabular}
\end{center}
Therefore, the homogeneous invariant ring 
$\mathscr R_{\gamma}=\bigoplus _{n\geq 0} A(Y_{\gamma})^{G_{\gamma}, \chi _{\gamma}^n}$ is 
generated by the elements of degree 1
\[
u_0=\xi_1\xi_2,\quad u_1=\xi _5,\quad u_2=\xi_6
\]
and the elements of degree 2
\[
\xi _3\xi _7,\; \xi _3\xi _8,\; \xi _4\xi _7,\; \xi _4\xi _8.
\]
Now we have to decide the relations between these generators. As we have a ring homomorphism
\[
\mathbb C[\xi_1, \dots, \xi_8] \twoheadrightarrow A(Y_{\gamma})^{SL(V)\times SL(Q)}
\hookrightarrow A(Y_{\gamma}), 
\]
the kernel of this morphism is the module of relations between $\xi _i$'s. As we are in 
a completely concrete situation, we can calculate this ideal using the elimination property 
of Gr\"obner basis implemented on computer algebra system. The author used {\sc Singular} \cite{S} 
to calculate the ideal. The result is that the relations are generated by
\begin{gather}
\xi_4\xi_5-\xi_3\xi_6,\; \xi_6\xi_7-\xi_5\xi_8,\notag\\
\xi_5\xi_6-\xi_3\xi_8,\; \xi_4\xi_7-\xi_3\xi_8,\label{rel:1112}\\
\xi_5^2 - \xi_3\xi_7,\; \xi_6^2-\xi_4\xi_8. \notag
\end{gather}
This immediately implies that we have an isomorphism of homogeneous rings
\[
\mathscr R_{\gamma}\cong \mathbb C[u_0,u_1,u_2], 
\]
therefore, $B_\gamma\cong \mathbb P^2$. 

For every point $\Psi=(\psi, (A,B))$, we have a point in the $GL(W)$-orbit of $\Psi$ such that 
\[
A=\begin{pmatrix} 0& 0 \\ a & 0\end{pmatrix},\quad 
B=\begin{pmatrix} 0& 0 \\ b & 0\end{pmatrix}, 
\]
for $A$ and $B$ commute. At such a point, $u_0$, $u_1$, and $u_2$ can be written as 
\[
\begin{aligned}
u_0 &=\begin{vmatrix} x_1 & x_2 \\ y_1 & y_2\end{vmatrix} \cdot
\begin{vmatrix} z_{11} & z_{12} \\ z_{21} & z_{22}\end{vmatrix},\\
u_1&=a\cdot\begin{vmatrix} x_1 & x_2 \\ z_{11} & z_{12}\end{vmatrix}\cdot
\begin{vmatrix} y_1 & y_2 \\ z_{11} & z_{12}\end{vmatrix},\quad
u_2=b\cdot\begin{vmatrix} x_1 & x_2 \\ z_{11} & z_{12}\end{vmatrix}\cdot
\begin{vmatrix} y_1 & y_2 \\ z_{11} & z_{12}\end{vmatrix}. 
\end{aligned}
\]
It immediately follows from this expression 
that the points in the orbit correspond to a strictly semistable sheaf 
exactly if $u_0=0$ or $u_1=u_2=0$ (cf. Proposition \ref{stability} (ii)). Therefore, 
\[
(B_{\gamma}\cap \Sigma)_{red}=(u_0=0)\cup (u_1=u_2=0)\subset \mathbb P^2 [u_0:u_1:u_2], 
\]
which is a disjoint union of a line and a point on $\mathbb P^2$. 

\subsection{}
Let $\gamma$ be of type $[2^2]$ and $W_1=\mathbb C^2$, $W_2=\mathbb C^2$. 
Our GIT setting in this case is
\[
\begin{aligned}
Y_{\gamma} &\cong \Hom (V,W_1)\times \Hom (V,W_2)\times N_2\times N_2,\\
G_{\gamma} & = GL(V)\times GL(W_1)\times GL(W_2),\\
\chi _{\gamma} &=(\det {}_V)^{-2}\cdot \det {}_{W_1}\cdot \det {}_{W_2}.
\end{aligned}
\]
Represent $\Psi\in Y_{\gamma}$ as
\[
\Psi=\Big( \psi = \left(\begin{array}{@{\,}cc@{\,}}
				x_{11} & x_{12} \\ 
				x_{21} & y_{22} \\ \hline
				y_{11} & y_{12}\\
				y_{21} & y_{22}
				\end{array}\right), \alpha = (A_1,B_1; A_2,B_2) \Big)
\]
by a coordinate. Then, $SL(V)$-invariants are generated by 
the entries of $A_1$, $B_1$, $A_2$, $B_2$, and
\[
f_1=\begin{vmatrix} x_{11} & x_{12} \\ x_{21} & x_{22} \end{vmatrix}, \;
f_2=\begin{vmatrix} y_{11} & y_{12} \\ y_{21} & y_{22} \end{vmatrix}, \;
w_1 =\begin{pmatrix}
	\begin{vmatrix} x_{11} & x_{12} \\ y_{11} & y_{12}\end{vmatrix} \vspace{6pt}\\
	\begin{vmatrix} x_{21} & x_{22} \\ y_{11} & z_{12}\end{vmatrix}\end{pmatrix},\;
w_2 =\begin{pmatrix}
	\begin{vmatrix} x_{11} & x_{12} \\ y_{21} & y_{22}\end{vmatrix} \vspace{6pt}\\
	\begin{vmatrix} x_{21} & x_{22} \\ y_{21} & z_{22}\end{vmatrix}\end{pmatrix}. 
\]
$SL(W_1)$ acts trivially on $f_1,f_2,A_2,B_2$, by left multiplication on $w_1, w_2$, and
by adjoint on $A_1,B_1$. Therefore, every $SL(V)\times SL(Q)$-invariant is a polynomial of
$\xi _1=f_1$, $\xi_2=f_2$, the entries of $A_2$, $B_2$, and
\[
\begin{aligned}
\xi _3&=\det (A_1w_1\mid w_1), & \xi_4&=\det (B_1w_1\mid w_1),\\
\xi _5&=\det (A_1 w_1\mid w_2), & \xi_6&=\det (B_1w_1\mid w_2),\\
\xi_7&=\det (A_1 w_2\mid w_2), & \xi_8&=\det (B_1w_2\mid w_2).
\end{aligned}
\]
It is straightforward to check that $GL(W_2)$ acts on the vectors
\[
\begin{pmatrix} \xi _3 \\ \xi _5 \\ \xi _7 \end{pmatrix},\quad 
\begin{pmatrix} \xi _4 \\ \xi _6 \\ \xi _8 \end{pmatrix}
\]
via $GL (\Sym ^2 W_2)$. Furthermore, we have an isomorphism of $\mathfrak{sl}(W_2)$-representations
\[
T:\Sym ^2 W_2\to \mathfrak{sl}(W_2), 
\]
under which 
\[
T\begin{pmatrix} \xi _3 \\ \xi _5 \\ \xi _7 \end{pmatrix}=\begin{pmatrix} \xi _5 & -\xi _3 \\ \xi _7 & -\xi _5\end{pmatrix}=: X_1,\quad
T\begin{pmatrix} \xi _4 \\ \xi _6 \\ \xi _8 \end{pmatrix}=\begin{pmatrix} \xi _6 & -\xi _4 \\ \xi _8 & -\xi _6\end{pmatrix}=: X_2.
\]
Therefore, a generating set of $SL(V)\times SL(W_1)\times SL(W_2)$-invariants 
is given by the matrix invariants for $A_2,B_2,X_1,X_2$, 
together with $f_1$ and $f_2$. The symbolic method tells us that the generating set of matrix invariants are given by 
traces of words in $A_2,B_2,X_1$, and $X_2$. Note the following properties:
\begin{itemize}
\item We can sort the word in trace, i.e., $\tr(XY)=\tr(YX)$.
\item We have $A_2^2=A_2B_2=B_2^2=O$. 
\item $X_1^2=(\xi _5^2-\xi _3\xi _7)\cdot I_2=O$ (see the relations \eqref{rel:1112}). Similarly, we have $X_1X_2=X_2^2=O$. 
\end{itemize}
Hence, the ring of $SL(V)\times SL(W_1)\times SL(W_2)$-invariants is generated by  
\[
\begin{aligned}
\zeta _1&=\xi_1=f_1, \quad \zeta _2=\xi _2=f_2,\\
\zeta _3&=\tr (A_2X_1), \quad \zeta _4=\tr (B_2X_1), \quad \zeta _5=\tr (A_2X_2), \quad \zeta _6=\tr (B_2X_2),
\end{aligned}
\]
and the table of the weights is as the following. 
\begin{center}
\begin{tabular}{c||c|c|c}
 & $\det {}_V$ & $\det {}_{W_1}$ & $\det {}_{W_2}$ \\ \hline
 $\zeta _1$ & -1 & 1 & 0\\
 $\zeta _2$ & -1 & 0 & 1\\
 $\zeta _3$ & -2 & 1 & 1\\
 $\zeta _4$ & -2 & 1 & 1\\
 $\zeta _5$ & -2 & 1 & 1\\
 $\zeta _6$ & -2 & 1 & 1\\ \hline
 $\chi_{\gamma}$ & -2 & 1 & 1 
\end{tabular}
\end{center}
We have an obvious relation $\zeta_4\zeta _5-\zeta _3\zeta _6=0$ and actually one can easily check that 
this is the only relation between $\zeta_i$'s. Summarizing, we get an isomorphism of graded rings
\[
\mathscr R_{\gamma}\cong \mathbb C [u_0, u_1, u_2, u_3, u_4] / (u_1u_4-u_2u_3), 
\]
where
\[
u_0=\zeta _1\zeta_2, \; u_1=\zeta _3,\; u_2=\zeta _4,\; u_3=\zeta _5,\; u_4=\zeta _6.
\]
Namely, $B_{\gamma}$ is isomorphic to a quadric cone in $\mathbb P^4$. 

At a point $\Psi=(\psi,(A_1,B_1; A_2,B_2))\in Y_{\gamma}$ such that 
\[
A_1=\begin{pmatrix}0 & 0 \\ a_1 & 0\end{pmatrix},\; 
B_1=\begin{pmatrix}0 & 0 \\ b_1 & 0\end{pmatrix},\; 
A_2=\begin{pmatrix}0 & 0 \\ a_2 & 0\end{pmatrix},\; 
B_2=\begin{pmatrix}0 & 0 \\ b_2 & 0\end{pmatrix},\; 
\]
$u_0, \cdots ,u_4$ are written as 
\[
\begin{aligned}
u _0 &=\begin{vmatrix} x_{11} & x_{12} \\ x_{21} & x_{22} \end{vmatrix} \cdot 
\begin{vmatrix} y_{11} & y_{12} \\ y_{21} & y_{22} \end{vmatrix}, \\
u_1  &= a_1a_2 D^2, \; 
u_2 = a_1b_2 D^2, \; 
u_3 = b_1a_2 D^2, \; 
u_4 = b_1b_2 D^2, 
\end{aligned}
\]
where $D=\begin{vmatrix} x_{11} & x_{12} \\ y_{11} & x_{12} \end{vmatrix}$. Noting that
\[
u_1=u_2=u_3=u_4=0 \Leftrightarrow 
D=0 \text{, or } a_1=b_1=0 \text{, or } a_2=b_2=0, 
\]
it is easy to see (cf. Proposition \ref{stability}, (ii)) that
\[
(B_{\gamma}\cap \Sigma)_{red}=(u_0=0)\cup (u_1=u_2=u_3=u_4=0)\subset B_{\gamma}
\subset \mathbb P^4[u_0:\cdots u_4], 
\]
namely, $(B_{\gamma}\cap \Sigma)_{red}$ is a disjoint union of a smooth hyperplane section and the vertex of $B_{\gamma}$. 

\subsection{} 
Let $\gamma$ be of type $[1,3]$ and $W=\mathbb C^3$. 
Then, we have 
\[
\begin{aligned}
Y_{\gamma} &\cong \Hom (V,\mathbb C)\times \Hom (V,W)\times N_3, \\
G_{\gamma} & = GL(V)\times \mathbb C^*\times GL(W),\\
\chi _{\gamma} &=(\det {}_V)^{-2}\cdot \id _{\mathbb C^*} \cdot \det {}_{W}.
\end{aligned}
\]
Represent $\Psi\in Y_{\gamma}$ as
\[
\Psi=\Big( \psi = \left(\begin{array}{@{\,}cc@{\,}}
				x_1 & x_2 \\ \hline 
				y_{11} & y_{12} \\ 
				y_{21} & y_{22}\\
				y_{31} & y_{32}
				\end{array}\right), \alpha = (A,B) \Big)
\]
using a coordinate. The ring of $SL(V)$-invariants is generated by the entries of $A$ and $B$, and
the Pl\"ucker coordinates
\[
\begin{aligned}
f_1 &=\begin{vmatrix}x_1 &x _2\\ y_{11} & y_{12}\end{vmatrix}, 
& f_2 &=\begin{vmatrix}x_1 &x _2\\ y_{21} & y_{22}\end{vmatrix}, 
& f_3 &=\begin{vmatrix}x_1 &x _2\\ y_{31} & y_{32}\end{vmatrix},\\
g_1 &=\begin{vmatrix}y_{11} &y _{12} \\ y_{21} & y_{22}\end{vmatrix}, 
& g_2 &=\begin{vmatrix}y_{11} &y _{12} \\ y_{31} & y_{32}\end{vmatrix}, 
& g_3 &=\begin{vmatrix}y_{21} &y _{22} \\ y_{31} & y_{32}\end{vmatrix}. 
\end{aligned}
\]
If we put them into a vector and an alternating form
\[
w=\begin{pmatrix} f_1 \\ f_2 \\ f_3\end{pmatrix}, \quad 
\omega = g_1 (e_1\wedge e_2)+g_2 (e_1 \wedge e_3)+g_3 (e_2\wedge e_3), 
\]
it is easy to verify that $SL(W)$ acts on the first vector by left multiplication and on the second form 
via $GL(\wedge ^2 W)$ by the induced representation. We have an isomorphism of 
$GL(W)$ representations 
\[
T: \wedge ^2 W\to W^{\vee}=\Hom (W, \mathbb C),
\]
under which 
\[
T(\omega) = \begin{pmatrix}g_3 & -g_2 & g_1\end{pmatrix}=:\lambda. 
\]
According to the symbolic method, the ring of mixed $SL(W)$-invariants of a vector, a co-vector, and matrices 
is given by
\begin{enumerate}[(a)]
\item $\lambda W_0(A,B) w$
\item $\tr (W_0(A,B))$
\item $\det (W_1(A, B)w \mid W_2(A, B)w\mid W_3(A,B)w)$ and
$\det  \left(\begin{array}{@{\,}cc@{\,}} \lambda W_1(A,B) \\ \hline \lambda W_2(A,B) \\ \hline \lambda W_3(A,B)
\end{array}\right)$, 
\end{enumerate}
where $w\in W$, $\lambda \in W^{\vee}$, $A,B\in \End (W)$, and 
$W_i(A,B)\; (i=0,1,2,3)$ stands for any word in $A$ and $B$. Note that we have 
the commutativity and the nilpotency conditions
\[
AB-BA=O, \; A^3=A^2B=AB^2=B^3=O.
\]
It is immediate that the invariants given by traces of words in $A$ and $B$ as in (b) vanish. 
For the type (a), the non-trivial invariants are 
\[
\begin{aligned}
\xi _1& =\lambda A w, & \xi _2 &=\lambda B w,\\
\xi _3 &=\lambda A^2 w, & \xi _4 &= \lambda AB w, & \xi _ 5 &= \lambda B^2 w
\end{aligned}
\]
(note that $\lambda w=0$, which is nothing but the Pl\"ucker relation). For the type (c), 
we have quite a few candidates. Note that $\det (W_1(A, B)w \mid W_2(A, B)w\mid W_3(A,B)w)=0$ 
if all of $W_i(A,B)$ are non-trivial, because $A$ and $B$ are nilpotent and commute each other. 
Similarly, a rank consideration of the endomorphisms $A^iB^j$ shows that 
the determinants of type (c) other than
\[
\begin{aligned}
\upsilon _1 & =\det (w\mid Aw\mid Bw),\\
\upsilon _2 & =\det (w\mid Aw\mid A^2w), &\upsilon _3&=\det (w\mid Aw\mid ABw), 
& \upsilon _4 &=\det (w\mid Aw\mid B^2w),\\
\upsilon _5 & =\det (w\mid Bw\mid A^2w), &\upsilon _6&=\det (w\mid Bw\mid ABw), 
& \upsilon _7 &=\det (w\mid Aw\mid B^2w)\\
\end{aligned}
\]
must vanish. We also have seven (possibly) non-trivial invariants 
$\zeta _1, \dots ,\zeta _7$ of type (c) involving $\lambda$ in the same way. 
These 19 invariants in total generate the invariant ring $A(Y_{\gamma})^{SL(V)\times SL(W)}$. 
The table of the weights of the invariants 
is the following. 
\begin{center}
\begin{tabular}{c||c|c|c c}
 & $\det {}_V$ & $ \id _{\mathbb C^*}$ & $\det {}_W$ \\ \hline
 $\xi _i$ & -2 & 1 & 1& $(i=1, \cdots, 5)$\\
 $\upsilon _i$ & -3 & 1 & 3 & $(i=1, \cdots, 7)$\\
 $\zeta _i$ & -3 & 2 & 0 & $(i=1, \cdots, 7)$\\ \hline
 $\chi_{\gamma}$ & -2 & 1 & 1 
\end{tabular}
\end{center}
Therefore, the homogeneous invariant ring $\mathscr R _{\gamma}$ is generated by
\[
\xi _i\;(i=1,\cdots 5)\quad \text{and}\quad \upsilon _i\cdot \zeta _j\; (i,j=1,\cdots ,7).
\]
Now we look at the relations among them. A Gr\"obner basis calculation tells us that 
all the invariants of the form $\upsilon _i\cdot \zeta _j$ are contained in the homogeneous subring 
generated by $\xi _1, \cdots, \xi _5$ and the only remaining relation is rather obvious
\[
\xi _4^2-\xi _3\xi _5=0. 
\]
This is equivalent to say that we have an isomorphism of homogenous rings
\[
\mathscr R_{\gamma}\cong \mathbb C[\xi _1, \cdots ,\xi _5] / (\xi _4^2-\xi _3\xi _5). 
\]
In other words, $B_{\gamma}$ is isomorphic to a rank 2 quadric in $\mathbb P^4$. 

\begin{claim}
A semistable point $\Psi\in Y_{\gamma}$ 
corresponds to a strictly semistable sheaf if and only if 
$\xi_3=\xi_4=\xi_5=0$
\end{claim}

\begin{proof}
Passing to a point in the $GL(W)$-orbit, we may assume that $\Psi=(\psi, (A,B))\in Y_{\gamma}$ 
is of the form  
\[
A=\begin{pmatrix}
0 & 0 & 0 \\
a_1 & 0 & 0\\
a_2 & a_3 & 0 \end{pmatrix},\quad 
B=\begin{pmatrix}
0 & 0 & 0 \\
b_1 & 0 & 0\\
b_2 & b_3 & 0 \end{pmatrix}. 
\]
If we evaluate $\xi _3, \xi _4, \xi _5$ at this point, we get
\[
\begin{aligned}
\xi _3&= a_1a_3f_1g_1=a_1a_3\begin{vmatrix} x_1 & x_2 \\ y_{11} & y_{12}\end{vmatrix}\cdot
\begin{vmatrix} y_{11} & y_{12} \\ y_{21} & y_{22}\end{vmatrix}\\
\xi _4&= a_1b_3f_1g_1=a_1b_3\begin{vmatrix} x_1 & x_2 \\ y_{11} & y_{12}\end{vmatrix}\cdot
\begin{vmatrix} y_{11} & y_{12} \\ y_{21} & y_{22}\end{vmatrix}\\
\xi _5&= b_1b_3f_1g_1=b_1b_3\begin{vmatrix} x_1 & x_2 \\ y_{11} & y_{12}\end{vmatrix}\cdot
\begin{vmatrix} y_{11} & y_{12} \\ y_{21} & y_{22}\end{vmatrix}\\
\end{aligned}
\]
According to Proposition \ref{stability}, we have to show that there exists a vector $v\in V$ such that
$\dim \Psi ((\mathbb Cv)\otimes \mathcal O_S)=2$. This is obvious if one of the two ($2\times 2$)-determinants 
vanishes. 

The remaining case is that $a_1a_3=a_1b_3=b_1b_3=0$, but none of the determinants vanish. 
This condition is equivalent to 
\[
a_1=b_1=0 \quad \text{or}\quad  a_3=b_3=0, 
\]
noting that $a_1b_3=a_3b_1$, for $AB-BA=O$. 
We may assume $x_2=0$ and $y_{11}=0$ by a base change on $V$. 
Assume $a_1=b_1=0$.  Then, 
\[
\Psi ((\mathbb C\begin{pmatrix}0 \\ 1\end{pmatrix})\otimes \mathcal O_S)
=\mathbb C \begin{pmatrix} y_{12} \\ y_{22} \\ y_{32}\end{pmatrix}
\oplus \mathbb C \begin{pmatrix} 0 \\ 0 \\ 1 \end{pmatrix} \subset W.
\]
On the other hand, if $a_3=b_3=0$, we have
\[
\Psi ((\mathbb C\begin{pmatrix}1 \\ 0\end{pmatrix})\otimes \mathcal O_S)
=\mathbb C \oplus \mathbb C \begin{pmatrix} 0 \\ y_{21} \\ y_{31} \end{pmatrix},  
\]
where the first $\mathbb C$ is the 1-dimensional $x$-coordinate space. 
This concludes that $\xi_3=\xi_4=\xi_5=0$ $\Rightarrow$ semistable. The converse is similar and even easier. 
\end{proof}

\subsection{}
Let us consider the last case, i.e., the case where $\gamma$ is of type $[4]$.
Let $W=\mathbb C^4$. We have 
\[
\begin{aligned}
Y_{\gamma} &\cong \Hom (V, W)\times N_4, \\
G_{\gamma} & = GL(V)\times \times GL(W),\\
\chi _{\gamma} &=(\det {}_V)^{-2}\cdot \det {}_{W}.
\end{aligned}
\]
Write $\Psi\in Y_{\gamma}$ as
\[
\Psi=\Big( \psi = \begin{pmatrix} x_{11} & x_{12}\\
                                                         x_{21} & x_{22}\\
                                                         x_{31} & x_{32}\\
                                                         x_{41} & x_{42}\end{pmatrix}, \alpha = (A,B) \Big)
\]
using a coordinate. As before, $SL(V)$-invarinants are generated by the Pl\"ucker coordinates 
and the entries of $A$ and $B$, 
and $SL(W)$ acts on the Pl\"ucker coordinates 
\[
p_{ij}=\begin{vmatrix} 
		x_{i1} & x_{i2}\\
		x_{j1} & x_{j2}\end{vmatrix}\quad (1\leqslant i<j\leqslant 4)
\]
via induced action on $\wedge ^2 W$. Note that $\wedge ^2 W\cong \mathbb C^6$ admits a non-degenerate 
symmetric form
\[
\langle \cdot, \cdot\rangle = \wedge : \wedge ^2 W \times \wedge ^2 W \to \wedge ^4 W\cong \mathbb C
\]
and we have a Lie algebra isomorphism
\[
T: \mathfrak{sl}(W)\overset{\sim}{\to} \mathfrak{so}(\wedge ^2W). 
\]
Therefore, we can translate the problem on $SL(W)$-invariants into a problem on $\mathfrak{so}(\wedge ^2W)$-invariants. 
Namely, we need to know the classical invariant theory on $SO(6)$-invariants. Again, the symbolic method 
(see \cite{PV} Theorem 9.3 and the table on p. 254) tells us the following:
\begin{quote}
Let $U=\mathbb C^{2m}$ with a non-degenerate symmetric form $\langle \cdot, \cdot\rangle$. 
Then, the ring of $SO(U)$-invariants in $\mathbb C[U\oplus \End (U)^{\oplus k}]$ is generated by
\begin{enumerate}[(a)]
\item $\langle u, W_0(X_1,\cdots ,X_k)u\rangle$, and
\item $\det (W_1(X_1,\cdots ,X_k)u,\cdots W_{2m}(X_1,\cdots ,X_k)u)$,
\end{enumerate} 
where $(u, (X_1,\cdots ,X_k))\in U\oplus \End (U)^{\oplus k}$, 
$W_i\; (i=0,\cdots ,2m)$ is any word in $X_1,\cdots ,X_k$.  
\end{quote}
Here, we have to be cautious about the fact that the isomorphism $T$ preserves the Lie bracket, 
but \emph{does not respect matrix products}. Namely, we have to include every possible non-zero words 
in $T(A^iB^j)\; (i,j\geqslant 0\text{ and }i+j<4)$, not only in $T(A)$ and $T(B)$, to get a complete generating set of invariants. 
But, we also note that we still have a reasonably bounded set of non-zero words, 
because all of $T(A^iB^j)$ are nilpotent and commute each other. 

Define $u=\displaystyle \sum _{1\leqslant i<j\leqslant 6} p_{ij}(e_i\wedge e_j)\in \wedge ^2 W$ and
\[
\begin{aligned}
\xi _1& =\langle u , T(A)^2 u\rangle, &
\xi _2& =\langle u , T(A)T(B) u\rangle, &
\xi _3& =\langle u , T(B)^2 u\rangle, \\
\xi _4& =\langle u , T(A)^4 u\rangle, &
\xi _5& =\langle u , T(A)^3T(B) u\rangle, &
\xi _6& =\langle u , T(A)^2T(B)^2 u\rangle,\\
\xi _7& =\langle u , T(A)T(B)^3 u\rangle, &
\xi _8& =\langle u , T(B)^4 u\rangle, \\
\xi _9& =\langle u , T(A^2)T(B) u\rangle, &
\xi _{10}& =\langle u , T(A)T(B^2) u\rangle. 
\end{aligned}
\]
One can check using a computer algebra system that the other non-zero invariants 
of type (a) and (b) above can be written as a polynomial of these $\xi _1,\cdots ,\xi _{10}$. 
This means that the invariant ring $A(Y_{\gamma})^{SL(V)\times SL(W)}$ is generated by
$\xi _i$'s. Moreover, the weight of $\xi _i$ is always the same as the weight of $\chi _{\gamma}$, so that 
the homogenous invariant ring $\mathscr R_{\gamma}$ is generated by 
the degree one part $\mathscr R_{\gamma,1}=\mathbb C\langle \xi _1, \cdots , \xi _{10}\rangle$. 

Now we examine the relations among $\xi _i$'s. According to a Gr\"obner basis calculation, again, 
the relations are given by 
\begin{gather}
\rank 
\begin{pmatrix} 
\xi _4 & \xi _5 & \xi _6 & \xi _7 & \xi _9\\
\xi _5 & \xi _6 & \xi _7 & \xi _8 & -\xi _{10}
\end{pmatrix}\leqslant 1, \label{[4]scroll}\\
\begin{aligned}\label{[4]hyp}
\xi _3&\xi _6 - 2\xi _2\xi _7 + \xi _1\xi _8 - \xi _{10}^2 =0,\\
\xi _3&\xi _5 - 2\xi _2\xi _6 + \xi _1\xi _7 + \xi _9\xi _{10}=0,\\
\xi _3&\xi _4 - 2\xi _2\xi _5 + \xi _1\xi _6 - \xi _9^2 =0.
\end{aligned}
\end{gather}
The determinantal equation \eqref{[4]scroll} is nothing but the defining equation of 
the singular rational scroll 
$\overline{\mathbb F}=\overline{\mathbb P}_{\mathbb P^1}(\mathcal O^{\oplus 3}\oplus \mathcal O(4)\oplus \mathcal O(1))
\subset \mathbb P^9$. The remaining three equations \eqref{[4]hyp} defines a subvariety of this scroll, 
which is isomorphic to $B_{\gamma}$. Take the natural proper birational morphism
\[
\mu : 
\mathbb F=\mathbb P_{\mathbb P^1}(\mathcal O^{\oplus 3}\oplus \mathcal O(4)\oplus \mathcal O(1))\to 
\overline{\mathbb F}
\]
and consider the strict transform $B'_{\gamma}=\mu _*^{-1}B_{\gamma}$. In terms of 
the bihomogeneous coordinate $(t_1,t_2; \zeta _1,\cdots ,\zeta _5)$ on $\mathbb F$, $\mu$ is described by 
\[
\begin{aligned}
\mu ^*\xi _1&=\zeta _1, & \mu ^*\xi _2&=\zeta _2, & \mu ^*\xi _3&=\zeta _3, \\
\mu ^*\xi _4&=t_1^4\zeta_4, & \mu ^*\xi _5&=t_1^3t_2\zeta_4, &
\mu ^*\xi _6&=t_1^2t_2^2\zeta_4, & \mu ^*\xi _7&=t_1t_2^3\zeta_4, & \mu ^*\xi _8&=t_2^4\zeta_4, \\
\mu ^*\xi _9&=t_1\zeta_5, & \mu ^*\xi _{10}&=-t_2\zeta_5. 
\end{aligned}
\]
Pulling back the equations \eqref{[4]hyp} by these relations, we know that $B'_{\gamma}$ is a 
zero locus of a single equation
\begin{equation}\label{[4]fiber}
(t_2^2\zeta _1 - 2t_1t_2\zeta _2 +t_1^2\zeta _3)\cdot \zeta _4-\zeta _5^2=0. 
\end{equation}
This implies $B'_{\gamma}$ restricted to a fiber $\mathbb P^4$ of $\mathbb F\to \mathbb P^1$ is 
a rank 2 quadric. In particular, $B_\gamma$ is a prime divisor in the singular scroll $\overline{\mathbb F}$. 
The image of the family of the vertex lines of the quadrics \eqref{[4]fiber} is contained in 
the plane $\Pi=\mathbb P^2 [\xi_1:\xi_2:\xi_3]$ and defined by 
\[
t_2^2\xi _1 - 2t_1t_2\xi _2 +t_1^2\xi _3=0\quad ([t_1:t_2]\in \mathbb P^1). 
\]
The envelope $\Delta$ of this family of lines on $\Pi$ is a smooth conic defined by
\[
\Delta : \xi_1\xi_3-\xi_2^2=0. 
\]

\begin{claim}
$(B_{\gamma}\cap \Sigma)_{red}=\Pi$, i.e., $\Psi\in Y_{\gamma}$ determines a strictly semistable sheaf 
if and only if $\xi _4=\cdots =\xi _{10}=0$. 
\end{claim}

\begin{proof}
We show that there exists a vector $v\in V$ such that
$\dim \Psi ((\mathbb Cv)\otimes \mathcal O_S)=2$ (Proposition \ref{stability} (ii)) 
if $\Psi$ is a semistable point satisfying $\xi _4=\cdots =\xi _{10}=0$. 

As before, we may assume that $\Psi = (\psi ,(A,B))\in Y_{\gamma}$ satisfies
\[
A=
\begin{pmatrix}
0 & 0 & 0 & 0\\
a_1 & 0 & 0 & 0\\
a_2 & a_3 & 0 & 0\\
a_4 & a_5 & a_6 & 0
\end{pmatrix},\quad 
B=
\begin{pmatrix}
0 & 0 & 0 & 0\\
b_1 & 0 & 0 & 0\\
b_2 & b_3 & 0 & 0\\
b_4 & b_5 & b_6 & 0
\end{pmatrix}.
\]
If we evaluate the functions $\xi _4,...,\xi _{10}$ at such a point, we get
\[
\begin{aligned}
\xi _4&=2 a_1 a_3^2 a_6 \cdot p_{12}^2, \quad
\xi _5=2 a_1 a_3 a_6b_3 \cdot p_{12}^2, \\
\xi _6 &=2 (a_1 a_3 b_3 b_6 + a_3 a_6 b_1 b_3) \cdot p_{12}^2, \\
\xi _7&=2 a_3 b_1 b_3 b_6 \cdot p_{12}^2,\quad 
\xi _8=2 b_1 b_3^2 b_6 \cdot p_{12}^2, \\ 
\xi _9&=(a_5 b_1 b_3 - a_3 b_1 b_5 - a_3 b_2 b_6 + a_2 b_3 b_6) p_{12}^2 + (a_6 b_1 b_3 - a_1 b_3 b_6) p_{12} p_{13}\\
\xi _{10} &=(a_3 a_6 b_2 - a_1 a_5 b_3 - a_2 a_6 b_3 + a_1 a_3 b_5 ) p_{12}^2 +(a_1 a_3 b_6 - a_3 a_6 b_1) p_{12} p_{13}
\end{aligned}
\]
If $p_{12}=\begin{vmatrix} x_{11} & x_{12} \\ x_{21} & x_{22}\end{vmatrix}=0$, we may assume 
$x_{12}=x_{22}=0$. Then, 
\[
\Psi ((\mathbb C\begin{pmatrix}0 \\ 1\end{pmatrix})\otimes \mathcal O_S)
=\mathbb C \begin{pmatrix} 0 \\ 0 \\ 1 \\ 0 \end{pmatrix}
\oplus \mathbb C \begin{pmatrix} 0 \\ 0 \\ 0\\ 1 \end{pmatrix}.
\]
So, let us assume $p_{12}\neq 0$. We may assume 
$\begin{pmatrix} x_{11} & x_{12} \\ x_{21} & x_{22}\end{pmatrix}=\begin{pmatrix} 1 & 0 \\ 0 & 1\end{pmatrix}$ 
by a coordinate change. Note that $AB-BA=O$ implies
\[
a_1:a_3:a_6=b_1:b_3:b_6.
\]
Thus, $\xi _4=\cdots =\xi _8=0$ implies
\[
a_1=b_1=0\quad \text{ or }\quad a_3=b_3=0\quad \text { or }\quad  a_6=b_6=0.
\]

First, we assume $a_3=b_3=0$. In this case, we automatically have $\xi_9=\xi _{10}=0$.  Then, 
\[
\Psi ((\mathbb C\begin{pmatrix}0 \\ 1\end{pmatrix})\otimes \mathcal O_S)
=\mathbb C \begin{pmatrix} 0 \\ 1 \\ x_{22} \\ x_{32} \end{pmatrix}
\oplus \mathbb C \begin{pmatrix} 0 \\ 0 \\ 0\\ 1 \end{pmatrix}.
\]

Next, we assume $a_6=b_6=0$. Then $AB-BA$ implies $a_3:a_5=b_3:b_5$, i.e., two vectors
\[
\begin{pmatrix}
0 \\ 0 \\ a_3 \\ a_5
\end{pmatrix}\quad \text{and} \quad
\begin{pmatrix}
0 \\ 0 \\ b_3 \\ b_5
\end{pmatrix}
\] 
are proportional to each other, and $\xi _9=\xi _{10}=0$ is again automatic. By the semistability of $\Psi$, 
one of these two vectors, say ${}^t (0,0,a_3, a_5)$, does not vanish. Then, 
\[
\Psi ((\mathbb C\begin{pmatrix}0 \\ 1\end{pmatrix})\otimes \mathcal O_S)
=\mathbb C \begin{pmatrix} 0 \\ 1 \\ x_{22} \\ x_{32} \end{pmatrix}
\oplus \mathbb C \begin{pmatrix} 0 \\ 0 \\ a_3\\ a_5 \end{pmatrix}. 
\]

Finally, assume $a_1=b_1=0$. In this case, we have $a_2:a_3:a_6=b_2:b_3:b_6$, so that
$\xi _9=\xi _{10}=0$. Then, 
\[
A\psi\begin{pmatrix} a_3 \\ -a_2\end{pmatrix} 
= \begin{pmatrix} 0 \\ 0 \\ 0 \\ * \end{pmatrix} =B\psi\begin{pmatrix} a_3 \\ -a_2\end{pmatrix}.
\]
Therefore, 
\[
\Psi ((\mathbb C\begin{pmatrix} a_3 \\ -a_2 \end{pmatrix})\otimes \mathcal O_S)
=\mathbb C \begin{pmatrix} a_3 \\ -a_2 \\ a_3x_{21}-a_2x_{22} \\ a_3x_{31}-a_2x_{32} \end{pmatrix}
\oplus \mathbb C \begin{pmatrix} 0 \\ 0 \\ 0\\ 1 \end{pmatrix}. 
\]
\end{proof}

\begin{claim}
$(B_\gamma \cap \Sigma ^1)_{red}=\Delta$.
\end{claim}

\begin{proof}
Thanks to the previous claim, we can restrict ourselves to the locus of $\Psi$ such that 
$\Ker \Psi=I_{Z_1}\oplus I_{Z_2}$, where $Z_i\in \Hilb ^2(S),\; \Supp (Z)=\{p\}\; (\gamma =4p)$. 
Namely, take $\Psi$ as
\[
\Psi=\Big( \psi = \begin{pmatrix} 1 & 0\\
                                                         0 & 0\\
                                                         0 & 1\\
                                                         0 & 0\end{pmatrix}, \big(
                                                       	A=\begin{pmatrix}
						0 & 0 & 0 & 0\\
						a_1 & 0 & 0 & 0\\
						0 & 0 & 0 & 0\\
						0 & 0 & a_6 & 0
						\end{pmatrix},
						B=\begin{pmatrix}
						0 & 0 & 0 & 0\\
						b_1 & 0 & 0 & 0\\
						0 & 0 & 0 & 0\\
						0 & 0 & b_6 & 0
						\end{pmatrix}\big) \Big). 
\]
Note that $\Ker \Psi$ is of the form $I_Z^{\oplus 2}$ exactly if $a_1:b_1=a_6:b_6$. 
If we evaluate $\xi_1$, $\xi _2$, and $\xi _3$ at such a point, 
\[
\xi _1=-2a_1a_6,\quad \xi _2=-a_1b_6-a_6b_1,\quad \xi _3=-2b_1b_6.
\]
Therefore, $a_1:b_1=a_6:b_6$ if and only if $\xi_1\xi_3-\xi _2^2=0$ (note that we are considering 
only the reduced structure). 
\end{proof}


\end{document}